\documentclass[3p,12pt]{elsarticle}
\usepackage{amsmath,amssymb,amsfonts,verbatim}
\journal{Journal of Algebra}

\newtheorem{theorem}{Theorem}[section]
\newtheorem{lemma}[theorem]{Lemma}
\newtheorem{prop}[theorem]{Proposition}
\newtheorem{cor}[theorem]{Corollary}

\newenvironment{proof}{{\flushleft\it Proof.}}{\hfill$\square$ \\}

\numberwithin{equation}{section}

\def\beq{\begin{equation}}
\def\eeq{\end{equation}}
\def\beqs{\begin{equation*}}
\def\eeqs{\end{equation*}}

\def\A{\mathcal{A\ts}}
\def\CC{\mathbb{C}}

\def\bt{\ts\,\raise-0.5pt\hbox{\small$\boxtimes$}\,\,}

\def\d{\partial}
\def\de{\delta}

\def\E{\mathcal{E}}
\def\End{\operatorname{End}\ts}
\def\ep{\varepsilon}

\def\F{\mathcal{F}}

\def\ge{\geqslant}
\def\gl{\mathfrak{gl}}
\def\glhat{\widehat{\mathfrak{gl}}}

\def\H{\mathfrak{H}}

\def\id{{\rm id}}
\def\Ind{\operatorname{Ind}}
\def\io{\iota}

\def\ka{\kappa}

\def\lcd{\ts,\ldots,}
\def\le{\leqslant}

\def\om{\omega}
\def\op{\oplus}
\def\ot{\otimes}

\def\p{\mathfrak{p}}

\def\q{\mathfrak{q}}

\def\R{\mathfrak{D}}

\def\si{\sigma}
\def\sl{\mathfrak{sl}}
\def\slhat{\widehat{\mathfrak{sl}}}
\def\Sym{\mathfrak{S}}

\def\T{\mathfrak{C}}
\def\th{\theta}
\def\ts{\hskip.75pt}
\def\tts{\hskip.5pt}

\def\U{\mathrm{U}}

\def\ZZ{\mathbb{Z}}

\begin{document}

\begin{frontmatter}

\title{On the functor of Arakawa, Suzuki and Tsuchiya}
 
\author[a,b]{Sergey Khoroshkin\,}
\address[a]{Institute for Theoretical and Experimental Physics, 
Moscow 117218, Russia}
\address[b]{National Research University Higher School of Economics, 
Moscow 101000, Russia}

\author[c]{Maxim Nazarov\,}
\address[c]{
University of York, 
York YO10 5DD, England
\\[20pt]
{\rm\normalsize 
To Professor Masatoshi Noumi on the occasion of his sixtieth birthday}
\\[-6pt]
}



\end{frontmatter}

\thispagestyle{empty}


\vspace{0pt}
\section*{Introduction}


\vspace{8pt}\noindent
The main aim 
of this article is to give detailed proofs
of basic properties of a certain functor introduced by Arakawa, Suzuki and
Tsuchiya in \cite{AST} and further studied by Suzuki in \cite{S2}. 
We denote this functor by $\A_N$ where
$N$ is any positive integer.
In our setting the functor $\A_N$
is applied to any module $V$ of the affine Lie algebra
$\slhat_m$ such that 
for any given vector in $V$, 
there is a degree $i$ such 
that the subspace $t^{\,i}\,\sl_m[t]\subset\slhat_m$ 
annihilates this vector.
Here $m$ is another positive integer and $t\,$ is a variable.
In the present article the Lie algebra 
$\slhat_m$ is regarded as a central
extension of the current Lie algebra
$\sl_m\ts[\,t,t^{-1}\ts]$ by a one-dimensional vector space
with a fixed basis element, which will be denoted by $C\,$. 

As a vector space, $\A_N(V)$ 
is the tensor product of $V$ with $N$ copies of the vector space
$\CC^{\ts m}[\ts t\ts,t^{-1}\ts]\,$.
The latter space 
can be regarded as 
a module of $\slhat_m$
where the central element $C$ acts as zero. Hence
$\A_N(V)$ is also a module of $\slhat_m\,$. 
It was proved in \cite{AST} that
the vector space $\A_N(V)$ comes with an action of
the trigonometric Cherednik algebra \cite{C1,C2,EG}.
We denote this algebra by $\T_N\,$.
The complex associative algebra $\T_N$ is generated by 
the symmetric group ring $\CC\ts\Sym_N\,$, 
the ring of Laurent polynomials in 
$N$ variables $x_1\lcd x_N$ and by another family of 
commuting elements $u_1\lcd u_N\,$.
The subalgebra of $\T_N$ generated by the first two rings is
the crossed product 
$\Sym_N\ltimes\CC[\ts x_1\ts,x_1^{-1}\lcd x_N\ts,x_N^{-1}\ts]$
where the 
group $\Sym_N$ permutes the $N$ variables. 
The subalgebra of $\T_N$ generated by 
$\Sym_N$ and $u_1\lcd u_N$ is the degenerate affine Hecke algebra
$\H_N$ introduced by Drinfeld \cite{D1} and Lusztig \cite{L}.
The other defining relations in $\T_N$ are 
\eqref{cross5},\eqref{cross6},\eqref{cross7}.
In particular, the algebra $\T_N$ depends on a parameter $\ka\in\CC\,$.

The action of the trigonometric Cherednik algebra $\T_N$ and that of
the affine Lie algebra $\slhat_m$ on the vector space $\A_N(V)$ 
do \emph{not} commute in general. However, let us 
suppose that
the 
element $C\in\slhat_m$ acts on $V$ as multiplication by the scalar $\ka-m\,$.
In other words, 
suppose that
the $\slhat_m\ts$-module $V$ is of level $\ka-m\,$.
Then the action of $\T_N$ on $\A_N(V)$ preserves the image
of the action of the subalgebra 
$t^{-1}\,\sl_m\ts[\,t^{-1}\ts]\subset\slhat_m\,$.
Hence the quotient vector space of $\A_N(V)$ by the image
inherits an action of 
$\T_N\,$.
This remarkable property of the functor $\A_N$ was also proved
in \cite{AST}. 
We give a more detailed 
proof of this property of $\A_N$ 
by following an approach of 
\cite{S2}.
At the same time we make Theorem 4.1 from \cite{S2} more precise.

\newpage

We will regard $\slhat_m$ as a subalgebra of the affine Lie algebra 
$\glhat_m\,$. The latter is a central extension of
the current Lie algebra
$\gl_m\ts[\,t,t^{-1}\ts]$ by the one-dimensional vector space
with the basis element $C\ts$. By definition, this one-dimensional 
vector space is contained
in the subalgebra $\slhat_m\subset\glhat_m\,$.
We start with any $\glhat_m\ts$-module 
such that for any given vector of the module, 
there is a degree $i$ such 
that the subspace $t^{\,i}\,\gl_m[t]\subset\glhat_m$ 
annihilates the vector.
Following \cite{S2}, we define an action of
the algebra $\T_N$ on the tensor product of this module
with $N$ copies of $\CC^{\ts m}[\ts t\ts,t^{-1}\ts]\,$.
These $N$ tensor factors are regarded as $\glhat_m\ts$-modules.
We prove that if the central element
$C$ acts on the initial $\glhat_m\ts$-module
as multiplication by $\ka-m\,$, 
then the action of $\T_N$ on the tensor product
preserves the image of the action of the subalgebra 
$t^{-1}\,\sl_m\ts[\,t^{-1}\ts]\subset\glhat_m\,$.
By modifying the action of $\T_N\ts$, 
we derive the above stated property of the functor 
$\A_N\,$; see our Section~3.

The functor $\A_N$ can be considered as the affine Lie algebra version
of another functor, first introduced by Cherednik \cite{C2} and further
studied in \cite{AST}. We denote this functor by $\F_N\,$,
and apply it to any module $U$ of the finite-dimensional 
Lie algebra $\sl_m\,$. As a vector space, $\F_N(U)$ 
is just the tensor product of $U$ with $N$ copies of the vector space
$\CC^{\ts m}$. By regarding the latter vector space as an
$\sl_m\ts$-module, we obtain an action of $\sl_m$ on $\F_N(U)\,$.
There is also an action of the degenerate affine Hecke algebra $\H_N$
on the vector space $\F_N(U)\,$, which commutes with the action of~$\sl_m\,$.
In particular, this construction was used by Suzuki \cite{S0} to prove a 
conjecture of Rogawski \cite{R} on the Jantzen filtration on the standard
$\H_N\ts$-modules.

In the present article we establish a connection between the functors
$\A_N$ and $\F_N\,$. For any $\sl_m\ts$-module $U$ we consider the
$\T_N\ts$-module induced from the $\H_N\ts$-module
$\F_N(U)\,$. Here $\H_N$ is regarded as a subalgebra of $\T_N\,$.
We demonstrate that for a certain $\slhat_m\ts$-module $V$,
the induced $\T_N\ts$-module is equivalent
to the quotient of $\A_N(V)$ by the image of the action of the 
subalgebra $t^{-1}\,\sl_m\ts[\,t^{-1}\ts]\subset\slhat_m\,$.
Namely $V$ is the $\slhat_m\ts$-module of level $\ka-m\,$,
parabolically induced from $U$ as a module 
over the subalgebra $\sl_m\subset\slhat_m\,$;
see our Section 2 for details.

We also consider the action of the algebra $\H_N$ on the tensor product of
any $\gl_m\ts$-module with $N$ copies of $\CC^{\ts m}$.
By modifying this action and regarding $\sl_m$ as a subalgebra of 
$\gl_m\,$, we get the above mentioned action of $\H_N$ on the 
space
$\F_N(U)$ for any $\sl_m\ts$-module $U$. 
Working with $\gl_m\ts$-modules also allows us to 
give an analogue of Theorem 2.1 from \cite{KN1},
where the role of $\H_N$ was played by
the Yangian of the Lie algebra $\gl_n$ with any $n\,$.
Namely, we consider the tensor product
of $N$ copies of $\CC^{\ts m}$ with a parabolically 
induced module of $\gl_m\,$.
Our Theorem \ref{1.25} describes the 
action of $\H_N$
on the quotient of the tensor product, taken relative
to the image of the action the nilpotent subalgebra of $\gl_m$
complementary to the parabolic subalgebra which the given module
of $\gl_m$ is induced from. The above mentioned result~of~\cite{KN1}
can be derived from our present Theorem \ref{1.25}
by using the Drinfeld functor~\cite{D1}.

The first named author was supported within the framework
of a subsidy granted to the Higher School of Economics  
by the Government of the Russian Federation to
implement the Global Competitiveness Program. He was also supported
by the RFBR grant 14-01-00547.
We are grateful to Tomoyuki Arakawa
and Takeshi Suzuki for very helpful conversations.
We dedicate this article to Masatoshi Noumi.
His works on the algebraic structures
arising from the theory of special functions have motivated
our interest in the results presented~here.  

\newpage


\section{Hecke algebras}
\label{sec:1}
\medskip


\subsection{}
\label{sec:11}

We begin with recalling a well known construction
from the representation theory of the 
\textit{degenerate affine Hecke algebra\/} $\H_N\ts$, which
corresponds to the general linear group $\mathrm{GL}_N$ over a local
non-Archimedean field. This algebra has been introduced by Drinfeld
[D2], see also the work of Lusztig [L]. 
By definition, the complex associative algebra $\H_N$ is
generated by the symmetric group algebra $\CC\ts\Sym_N$ and by the pairwise
commuting elements $u_1\lcd u_N$ with the cross relations for $p=1\lcd N-1$
and $q=1\lcd N$
\begin{align}
\label{cross1}
\si_{p}\,u_q&=u_q\,\si_{p}
\quad\text{for}\quad
q\neq p\ts,p+1\ts;
\\
\label{cross2}
\si_{p}\,u_p&=u_{p+1}\,\si_{p}-1\ts.
\end{align}
Here and in what follows $\si_p\in \Sym_N$ denotes the 
transposition of numbers $p$ and $p+1\ts$. 
More generally, $\si_{pq}\in \Sym_N$ will denote the 
transposition of the numbers $p$ and $q\ts$. The group algebra
$\CC\ts\Sym_N$ can be then regarded as a subalgebra in $\H_N\ts$.
Furhtermore, it follows from the defining relations of $\H_N$
that a homomorphism $\H_N\to\CC\ts\Sym_N\ts$, identical
on the subalgebra $\CC\ts\Sym_N\subset \H_N\ts$,
can be defined by the assignments
\begin{equation}
\label{evalhecke}
u_p\mapsto\si_{1p}+\ldots+\si_{p-1,p}
\quad\text{for}\quad
p=1\lcd N.
\end{equation}

We will also use the elements of the algebra $\H_N\ts$
\begin{equation}
\label{zp}
z_p=u_p-\si_{1p}-\ldots-\si_{p-1,p}
\quad\text{where}\quad
p=1\lcd N.
\end{equation}
Notice that $z_p\mapsto0$ under the homomorphism $\H_N\to\CC\ts\Sym_N\ts$
defined by \eqref{evalhecke}. For every permutation $\si\in\Sym_N$ we have
\begin{equation}
\label{yrel}
\si\,z_p\,\si^{-1}=z_{\ts\si(p)}\,.
\end{equation}
It suffices to verify \eqref{yrel} when
$\si=\si_q$ and $q=1\lcd N-1$. Then (\ref{yrel})
is equivalent to the relations \eqref{cross1},\eqref{cross2}.
The elements $z_1\lcd z_N$ do not commute, but satisfy
the commutation relations
\begin{equation}
\label{ycom}
[\,z_p\,,z_q\,]=\si_{pq}\,(z_p-z_q)\ts.
\end{equation}
Let us verify the equality in \eqref{ycom}.
Both sides of \eqref{ycom} are antisymmetric in $p$ and $q\ts$,
so it suffices to consider only the case when $p<q\ts$. Then
\begin{align*}
&[\,z_p\,,u_q\,]=[\,u_p-\si_{1p}-\ldots-\si_{p-1,p}\,,u_q\,]=0\ts,
\\[2pt]
&[\,z_p\,,z_q\,]=[\,z_p\,,z_q-u_q\,]=
-\ts[\,z_p\,,\si_{1q}+\ldots+\si_{q-1,q}\,]=
-\ts[\,z_p\,,\si_{pq}\,]=\si_{pq}\,(z_p-z_q)
\end{align*}
where we used \eqref{yrel}.
Obviously, the algebra $\H_N$ is generated by $\CC\ts\Sym_N$
and the elements $z_1\lcd z_N\ts$.
Together with relations in $\CC\ts\Sym_N$, 
\eqref{yrel} and \eqref{ycom} are defining relations for $\H_N\ts$.

The construction that we wish to recall now is due to Cherednik
\cite[Example 2.1]{C2}. It has been further developed by
Arakawa, Suzuki and Tsuchiya \cite[Subsection 5.3]{AST}. 
Let $U$ be any module over the complex general linear Lie algebra $\gl_m\ts$.
Let $E_{ab}\in\gl_m$ with $a,b=1\lcd m$ be the
standard matrix units. We will also regard the matrix units $E_{ab}$
as elements of the algebra $\End(\CC^{\ts m})$, this should not cause any 
confusion. Let us consider the tensor product $(\CC^{\ts m})^{\ot N}\ot\,U$
of $\gl_m\ts$-modules. Here each of the $N$ tensor factors $\CC^{\ts m}$ is
a copy of the natural $\gl_m$-module. We shall use the indices
$1\lcd N$ to label these $N$ tensor factors.
For any index $p=1\lcd N$
we will denote by $E^{\ts(p)}_{ab}$ the operator on the vector space 
$(\CC^{\ts m})^{\ot N}$ acting as
\begin{equation}
\label{epab}
\id^{\ts\ot\ts(p-1)}\ot E_{ab}\ot\id^{\ts\ot\ts(N-p)}\,.
\end{equation}

\begin{prop}
\label{1.1}
{\rm(i)}
Using the\/ $\gl_m\ts$-module structure of $U$,  
an action of the algebra $\H_N$ on the vector space 
$(\CC^{\ts m})^{\ot N}\ot\ts U$
is defined as follows: the symmetric 
group $\Sym_N\subset \H_N$ acts 
by permutations of the $N$ tensor factors\/ $\CC^{\ts m}$, and the element 
$z_p\in \H_N$ with $p=1\lcd N$~acts~as 
\begin{equation}
\label{yact}
\sum_{a,b=1}^m E_{ab}^{\ts(p)}\ot E_{\ts ba}\,.
\end{equation}
{\rm(ii)} 
This action of $\H_N$ commutes with the
{\rm(}diagonal\,{\rm)} action of\/ 
$\gl_m$ on $(\CC^{\ts m})^{\ot N}\ot\ts U\ts$.
\end{prop}

To prove this proposition one only needs to verify that the
commutation relations \eqref{ycom} are satisfied by 
the operators \eqref{yact} with $p=1\lcd N$ instead of
the elements $z_1\lcd z_N\in\H_N$ respectively.
This verification is straightforward, see
for instance in \cite[Section 1]{KN1}.
By using Proposition \ref{1.1} we get a functor
$
\E_N: U\mapsto(\CC^{\ts m})^{\ot N}\ot U
$
from the category of all $\gl_m\ts$-modules 
to the category of bimodules over $\gl_m$ and $\H_N\ts$.

Now let $\sl_m\subset\gl_m$ be the complex special linear Lie algebra.
We will also use a version of Proposition \ref{1.1} 
for the vector space $(\CC^{\ts m})^{\ot N}\ot\,U$ 
where 
$U$ is a module not of $\gl_m$ but
only of $\sl_m\,$. Denote $I=E_{11}+\ldots+E_{mm}$ 
so that $\gl_m=\sl_m\op\CC\,I\,$.
Moreover
\begin{equation}
\label{i}
\sum_{a,b=1}^m 
E_{ab}\ot E_{\ts ba}\,\in\,\frac1m\,\,I\ot I\,+\,\sl_m\ot\sl_m\,.
\end{equation}
Therefore an action of 
\begin{equation}
\label{yacts}
\sum_{a,b=1}^m E_{ab}^{\ts(p)}\ot E_{\ts ba}-\frac1m\ \id^{\,\ot\ts N}\ot I
\end{equation}
can be defined on the vector space $(\CC^{\ts m})^{\ot N}\ot\,U$ 
by using only the $\sl_m\ts$-module structure of $U$.
Because the element $I\in\gl_m$ is central, 
the operators \eqref{yacts} with $p=1\lcd N$ satisfy the 
same commutation relations \eqref{ycom}
as the operators \eqref{yact} respectively instead of $z_1\lcd z_N$.
  
\begin{cor}
\label{1.2}
{\rm(i)}
Using the\/ $\sl_m\ts$-module structure of $U$,  
an action of the algebra $\H_N$ on the vector space 
$(\CC^{\ts m})^{\ot N}\ot\ts U$
is defined as follows: the 
group $\Sym_N\subset \H_N$ acts 
by permutations of the $N$ tensor factors\/ $\CC^{\ts m}$, and the element 
$z_p\in \H_N$ with $p=1\lcd N$~acts~as \eqref{yacts}.
\\
{\rm(ii)} 
This action of $\H_N$ commutes with the
{\rm(}diagonal\,{\rm)} action of\/ 
$\sl_m$ on $(\CC^{\ts m})^{\ot N}\ot\ts U\ts$.
\end{cor}

Hence we get a functor
$
\F_N: U\mapsto(\CC^{\ts m})^{\ot N}\ot U
$
from the category of all \text{$\sl_m\ts$-modules} 
to the category of bimodules over $\sl_m$ and $\H_N\ts$.
Our main subject will be an analogue of
this functor for the affine Lie algebra $\slhat_m$ instead of 
$\sl_m\,$. The role of the degenerate affine Hecke algebra 
$\H_N$ will be then played by the 
\text{trigonometric Cherednik algebra\/}~$\T_N\,$.


\subsection{}
\label{sec:115}

For any $f\in\CC$ an automorphism of the degenerate affine Hecke algebra 
$\H_N$ identical on the subalgebra $\CC\ts\Sym_N\subset\H_N$ can be defined 
by mapping $u_p\mapsto u_p+f$ for all $p=1\lcd N$.
Hence we can modify the functor $\E_N$ by pulling its defining action
of $\H_N$ back through this automorphism. We will denote by $\E_N^{\,f}$
the modified functor. Like $\E_N=\E_N^{\,0}\,$, this is a functor
$
U\mapsto(\CC^{\ts m})^{\ot N}\ot U
$
from the category of all $\gl_m\ts$-modules 
to the category of bimodules over $\gl_m$ and $\H_N\ts$.
The modified functor will be needed to state Theorem \ref{1.25} below.
We will also let the parameter $m$ of the target category 
of $\E_N^{\,f}$ vary. This should cause no confusion.

Let $n$ be another positive integer.
The decomposition $\CC^{\ts m+n}=\CC^{\ts m}\op\CC^{\ts n}$ 
determines an embedding
of the direct sum $\gl_m\op\gl_n$ of Lie algebras to $\gl_{\ts m+n}\,$.
As a subalgebra of $\gl_{\ts m+n}\,$,
the direct summand $\gl_m$ is spanned by the matrix units 
$E_{ab}\in\gl_{\ts m+n}$ where
$a,b=1\lcd m\ts$. The direct summand $\gl_n$ is spanned by
the matrix units $E_{ab}$ where $a,b=m+1\lcd m+n\ts$.
Let $\q$ be 
the Abelian subalgebra of $\gl_{\ts m+n}$
spanned by the elements $E_{\ts ab}$ 
for all $a=m+1\lcd m+n$ and $b=1\lcd m\ts$.
Let $\p$ be the subalgebra of $\gl_{\ts m+n}$
spanned by those matrix units which do not belong to $\q\,$, so that 
we have $\gl_{m+n}=\p\op\q\,$.
Then $\p$ is a maximal parabolic subalgebra of the reductive 
Lie algebra $\gl_{\ts m+n}\,$. 
Note that $\gl_m\op\gl_n\subset\p$ by definition.

Now let $V$ be any $\gl_n\ts$-module.
Denote by $U\bt V$ the $\gl_{\ts m+n}$-module
\textit{parabolically induced\/}
from the $\gl_m\op\gl_n\ts$-module $U\ot V$.
To define $U\bt V$, one first extends the action of the Lie algebra
$\gl_m\op\gl_n$ on $U\ot V$ to 
$\p$ so that 
any matrix unit in $\p$ complementary to $\gl_m\op\gl_n$ 
acts on $U\ot V$ as zero.
By definition, $U\bt V$ is the $\gl_{\ts m+n}\ts$-module induced from the 
$\p\ts$-module $U\ot V$. Theorem~\ref{1.25}
will provide a description of the space $\E_N(\ts U\bt V\ts)_{\ts\q}$ 
of \textit{coinvariants} of $\E_N(\ts U\bt V\ts)$ 
relative to the action of the subalgebra $\q\subset\gl_{\ts m+n}\,$.
This space is the quotient of the vector space
$\E_N(\ts U\bt V\ts)$ by the image of the action of $\q\,$. 
Note that here the functor $\E_N$ is applied to 
a module of $\gl_{\ts m+n}$ rather than of $\gl_m\,$.
This is clear by the notation. Hence
we have an action of $\H_N$ on the vector space
$\E_N(\ts U\bt V\ts)_{\ts\q}\,$.
The subalgebra 
$\gl_m\op\gl_n\subset\gl_{m+l}$ also acts on this vector space,
and the latter action commutes with that of the algebra $\H_N\ts$.

For $K=0,1\lcd N$ let $\Sym_{K,N-K}$ be the subgroup of the symmetric 
group $\Sym_N$ preserving the subset $\{1\lcd K\}\subset\{1\lcd N\}\ts$.
This subgroup is naturally isomorphic to the direct product
$\Sym_K\times\Sym_{N-K}\ts$.
Further, the tensor product $\H_K\ot\H_{N-K}$ can be naturally identified 
with the subalgebra of $\H_N$ generated by the subgroup
$\Sym_{K,N-K}\subset\Sym_N$ and by all 
$u_1\lcd u_N\ts$. Denote by $\H_{K,N-K}$ this subalgebra.
We have the usual induction functor $\Ind_{\,\H_{K,N-K}}^{\,\H_N}\,$.

\begin{theorem}
\label{1.25}
The bimodule $\E_N(\ts U\bt V\ts)_{\ts\q}$ of\/ $\gl_m\op\gl_n$ and\/ $\H_N$
is equivalent to 
\begin{equation}
\label{inde}
\mathop{\op}\limits_{K=0}^N\,
\Ind_{\,\H_{K,N-K}}^{\,\H_N}\E_K(U)\ot\E_{N-K}^{\,-\ts m}(V)\,.
\end{equation}
\end{theorem}

\begin{proof}
For any index $K=0\lcd N$ the vector space of the corresponding
direct summand in \eqref{inde} is the same as that of the induced
$\Sym_N\ts$-module
$$
\Ind_{\,\Sym_{K,N-K}}^{\,\Sym_N}
(\CC^{\ts m})^{\ot K}\ot U\ot(\CC^{\ts n})^{\ot\ts(N-K)}\ot V\tts.
$$
By definition, the latter $\Sym_N\ts$-module is a quotient
of the vector space
\begin{equation}
\label{quot}
\CC\ts\Sym_N\ot
(\CC^{\ts m})^{\ot K}\ot U\ot(\CC^{\ts n})^{\ot\ts(N-K)}\ot V
\end{equation}
such that the right multiplication by any element
of the subgroup $\Sym_{K,N-K}\subset\Sym_N$ in the first tensor factor
of \eqref{quot} has the same effect on the quotient,
as the corresponding permutation of the $K$ tensor factors $\CC^{\ts m}$
and the $N-K$ tensor factors $\CC^{\ts n\ts}$.
The action of the group $\Sym_N$ on the quotient is 
then via the left multiplication in the first tensor factor
of \eqref{quot}. 

By letting the symmetric group $\Sym_N$ permute
all $N$ tensor factors $\CC^{\ts m}$ and $\CC^{\ts n}$ of
\eqref{quot}, the direct sum over $K=0\lcd N$ of the above described
quotients can be identified with 
\begin{equation}
\label{mn}
(\CC^{\ts m+n})^{\ot N}\ot U\ot V\tts.
\end{equation}
Here we use the decomposition 
$\CC^{\ts m+n}=\CC^{\ts m}\op\CC^{\ts n}$
and transpose the tensor factor $U$ of \eqref{quot} with each of
the next $N-K$ tensor factors $\CC^{\ts n}\ts$. 
Under this identification,
the action of $\Sym_N$ on the quotients becomes 
the permutational 
action on the first $N$ tensor factors of \eqref{mn}.
To avoid confusion, we will denote respectively by
$A$ and $B$ the subspaces of $\CC^{\ts m+n}$ corresponding to
the first and the second summands
in the decomposition $\CC^{\ts m+n}=\CC^{\ts m}\op\CC^{\ts n}\ts$.

Let us describe the action of the elements $u_1\lcd u_N\in\H_N$ on
\eqref{mn} coming from the identification of
this vector space with that of the bimodule \eqref{inde}.
Under this identification, for any $K$ the subspace
\begin{equation}
\label{subK}
A^{\ts\ot K}\ot B^{\,\ot\ts(N-K)}\ot U\ot V
\end{equation}
of \eqref{mn} is preserved by the action of 
$u_1\lcd u_N\,$. 
For $p\le K$ the element $u_p$ acts on the subspace
\eqref{subK} as
\begin{equation}
\label{uk}
\sum_{q=1}^p\,\si_{pq}\ot\id\ot\id\,+
\sum_{a,b=1}^m E_{ab}^{\ts(p)}\ot E_{\ts ba}\ot\id\,.
\end{equation}
Here $\si_{pq}$ is the permutation of the $p\,$th and
$q\,$th tensor factors $\CC^{\ts m+n}$ of \eqref{mn}
while $E_{ab}^{\ts(p)}$ is the operator on 
$(\CC^{\ts m+n})^{\ot N}$ acting as \eqref{epab}.
For $p>K$ the element $u_p$ acts on 
\eqref{subK} as
\begin{equation}
\label{unk}
\sum_{q=K+1}^p\si_{pq}\ot\id\ot\id\,+
\!\!\!\sum_{a,b=m+1}^{m+n}\!\!\!
E_{ab}^{\ts(p)}\ot\id\ot E_{\ts b-m\ts,\ts a-m}\,-\,m\,.
\end{equation}
The union of the subspaces \eqref{subK} for all $K=0\lcd N$
is cyclic in \eqref{mn} under the action of the subalgebra
$\CC\ts\Sym_N\subset\H_N\,$. Hence the action of $\H_N$ on
\eqref{mn} is now uniquely determined.

On the other hand, 
the 
vector space of the $\gl_{\ts m+n}\ts$-module $U\bt V$
can be identified with 
$\U(\q)\ot U\ot V$ whereon the Lie subalgebra 
$\q\subset\gl_{\ts m+n}$ acts via
left multiplication in the first tensor factor. 
Here the Lie algebra $\q$ is Abelian.
Hence its universal enveloping algebra
$\U(\q)$ is a free commutative algebra over $\CC$
generated by the elements $E_{ab}$ with $a=m+1\lcd m+n$ and $b=1\lcd m\ts$.
The vector space of the $\gl_m\op\gl_n\ts$-module $U\ot V$ 
can be then identified with the subspace
\begin{equation}
\label{subuv}
1\ot U\ot V\subset \U(\q)\ot U\ot V\ts.
\end{equation}
By the definition of a parabolically induced module, 
the action of the subalgebra $\p\subset\gl_{\ts m+n}$
preserves this subspace.
The $E_{ab}\in\p$ with $a,b=1\lcd m$ and $a,b=m+1\lcd m+n$
act on this subspace as
$
\id\ot E_{ab}\ot\id
\ \text{and}\ 
\id\ot\id\ot E_{a-m\ts,\ts b-m}
$
respectively,
while any element $E_{ab}\in\p$ with $a=1\lcd m$ and $b=m+1\lcd m+n$
acts as zero.
All this determines the action of $\gl_{\ts m+n}$ on $\U(\q)\ot U\ot V\ts$.

Now consider the bimodule $\E_N(\ts U\bt V\ts)$ 
over $\gl_{\ts m+n}$ and $\H_N\,$.
Its vector space is
\begin{equation}
\label{mnq}
W=(\CC^{\ts m+n})^{\ot N}\ot\ts \U(\q)\ot U\ot V\ts.
\end{equation}
The corresponding space of $\q\ts$-coinvariants is isomorphic to \eqref{mn}.
Indeed, we can define a bijective linear mapping $\io$ from
\eqref{mn} to the quotient 
$W/\ts\q\,W$ as follows.
First we map \eqref{mn} to the subspace
$$
(\CC^{\ts m+n})^{\ot N}\ot\ts1\ot U\ot V
\subset W
$$
in the natural way, and then regard the image of the latter mapping 
modulo $\q\,W$. 
To prove the bijectivity of resulting mapping 
$\io\,$, consider the ascending $\ZZ\ts$-filtration on the vector space
$$
(\CC^{\ts m+n})^{\ot N}=(\ts A\op B\ts)^{\ot N}
$$
defined by the tensor degree in $A\ts$.
The action of the subalgebra $\q\subset\gl_{\ts m+n}$ preserves
the filtration, and the corresponding graded action of $\q$ is trivial.
Therefore this filtration on $(\CC^{\ts m+n})^{\ot N}$
induces an ascending filtration on $W$ such that the
corresponding graded action of $\q$ is via left multiplication in the
tensor factor $\U(\q)$ of \eqref{mnq}. The bijectivity of 
our $\io$ now follows, because the algebra $\U(\q)$ is free commutative.  

The mapping $\io$ defined above is $\gl_m\op\gl_n\ts$-equivariant,
see the above description of the action of $\p$ on the subspace 
\eqref{subuv}. This $\io$ is also equivariant relative
to the action of $\Sym_N\ts$. But relative to the action of $\Sym_N$
on \eqref{mn}, the union of the subspaces
\eqref{subK} for $K=0\lcd N$ is cyclic.
To complete the proof of Theorem \ref{1.25} 
it now remains to check for each $p=1\lcd N$ the 
$u_p\ts$-equivariance of the restriction of the mapping
$\io$ to each of these subspaces.

The image of the subspace \eqref{subK} under $\io$ is the subspace
\begin{equation}
\label{1subK}
A^{\ts\ot K}\ot B^{\,\ot\ts(N-K)}\ot 1\ot U\ot V\subset W
\end{equation}
regarded modulo $\q\,W$. The elements $E_{ab}\in\gl_{\ts m+n}$
with $b=m+1\lcd m+n$ annihilate the subspace $X\subset \CC^{\ts m+n}\ts$.
The elements $E_{ba}\in\gl_{\ts m+n}$ with $b=1\lcd m$ and 
$a=m+1\lcd m+n$ annihilate the subspace \eqref{subuv}. Hence for $p\le K$
the element $u_p\in\H_N$ acts on the subspace \eqref{1subK} as 
$$
\sum_{q=1}^p\,\si_{pq}\ot\id\ot\id\ot\id\,+
\sum_{a,b=1}^m E_{ab}^{\ts(p)}\ot\id\ot E_{\ts ba}\ot\id\,.
$$
By comparing the above displayed sum with \eqref{uk} we conclude that
the restriction of the mapping $\io$ to the subspace \eqref{subK}
is $u_p\ts$-equivariant for any $p\le K\ts$.

Now suppose that $p>K\ts$.
The elements $E_{ab}\in\gl_{\ts m+n}$
with $b=1\lcd m$ annihilate the subspace $B\subset \CC^{\ts m+n}\ts$.
Hence for our $p$
the element $u_p\in\H_N$ acts on the subspace \eqref{1subK} as
\begin{gather}
\nonumber
\sum_{q=1}^p\,\si_{pq}\ot\id\ot\id\ot\id\,+
\sum_{a=1}^m
\sum_{b=m+1}^{m+n} E_{ab}^{\ts(p)}\ot E_{\ts ba}\ot\id\ot\id\ +
\\
\label{lastsum}
\!\!\!\!
\sum_{a,b=m+1}^{m+n}\!\!\!
E_{ab}^{\ts(p)}\ot\id\ot\id\ot E_{\ts b-m\ts,\ts a-m}
\,.
\end{gather}
The tensor factor $E_{ba}$ in the display 
\eqref{lastsum} belongs to the subalgebra $\q\subset\gl_{\ts m+n}\,$.
Therefore modulo $\q\,W$, the result of applying
the first line of the display \eqref{lastsum}
to elements of $W$ is the same as the result of applying
\begin{equation}
\label{secondlast}
\sum_{q=1}^p\,\si_{pq}\ot\id\ot\id\ot\id\,-\,
\sum_{q=1}^N\,\,\ts
\sum_{a=1}^m
\sum_{b=m+1}^{m+n} E_{ba}^{\ts(q)}E_{ab}^{\ts(p)}\ot\id\ot\id\ot\id\,.
\end{equation}
Here the tensor factors 
$E_{ba}^{\ts(q)}E_{ab}^{\ts(p)}$ with $q>K$ but with $q\neq p$ 
vanish on the subspace 
\begin{equation}
\label{xy}
A^{\ts\ot K}\ot B^{\,\ot\ts(N-K)}\subset(\CC^{\ts m+n})^{\ot N}\ts.
\end{equation}
Any tensor factor $E_{ba}^{\ts(q)}E_{ab}^{\ts(p)}$ in \eqref{secondlast}
with $q\le K$
acts on this subspace as permutation~$\si_{pq}\,$. 
The sum of the tensor factors 
$$
E_{ba}^{\ts(p)}E_{ab}^{\ts(p)}=E_{bb}^{\ts(p)}
$$ 
over 
$a=1\lcd m$ and $b=m+1\lcd m+n$ acts on 
the subspace \eqref{xy} as the scalar $m\,$. 
Hence after cancellations,
the sum \eqref{secondlast} acts on the subspace \eqref{1subK} as
\begin{equation}
\label{thirdlast}
\sum_{q=K+1}^p\,\si_{pq}\ot\id\ot\id\ot\id\,-\,m\,.
\end{equation}
By comparing \eqref{unk} with the sum of the second line in 
\eqref{lastsum} and of \eqref{thirdlast} 
we now conclude that
the restriction of the mapping $\io$ to the subspace \eqref{subK}
is $u_p\ts$-equivariant for $p>K\ts$.
\end{proof}


\section{Cherednik algebras}
\medskip


\subsection{}
\label{sec:12}

First define the \textit{rational Cherednik algebra\/} $\R_N\,$.
This is a complex associative algebra depending on a parameter $\ka\in\CC\,$.
It is generated by the 
symmetric group algebra $\CC\ts\Sym_N$ together with two sets of 
commuting elements $x_1\lcd x_N$ and $y_1\lcd y_N$ where like in \eqref{yrel}
\begin{equation}
\label{sirel}
\si\,x_p\,\si^{-1}=x_{\ts\si(p)}
\quad\text{and}\quad
\si\,y_p\,\si^{-1}=y_{\ts\si(p)}
\end{equation}
for every permutation $\si\in\Sym_N\ts$. For any indices $p$ and $q\,$,
the elements $y_p$ and $x_q$ do not commute with each other but satisfy
the commutation relations
\begin{align}
\label{cross3}
[\,y_{p}\,,x_q\,]&=-\,\si_{pq}
\quad\text{for}\quad
q\neq p\,;
\\
\label{cross4}
[\,y_{p}\,,x_p\,]&=\ka+\sum_{r\neq p}\,\si_{pr}\,.
\end{align}

Multiplication in the algebra $\R_N$ provides a bijective linear map
\begin{equation}
\label{ratmult}
\CC[\ts x_1\lcd x_N\ts ]\ot
\CC\ts\Sym_N\ot\CC[\ts y_1\lcd y_N\ts]
\,\to\,\R_N\,;
\end{equation}
see for instance \cite[Theorem 1.3]{EG}. The bijectivity here
also follows from the next proposition, which has been a motivation
for introducing the algebra $\R_N\,$. For any $p=1\lcd N$
consider the \textit{Dunkl operator}
acting on the vector space $\CC[\ts x_1\lcd x_N\ts]$ as
\begin{equation}
\label{dp}
\ka\,\d_p+\sum_{r\neq p}\,\frac1{x_p-x_r}\,(1-\si_{pr})\,.
\end{equation}
Here $\d_p$ means the derivation in the polynomial ring $\CC[x_1\lcd x_N]$
relative to the variable~$x_p\,$, and
the symmetric group $\Sym_N$ acts on the latter ring
by permutations of $x_1\lcd x_N$ as usual.

\begin{prop}
\label{1.3}
\label{dunkl} 
An action of the algebra 
$\R_N$ on the vector space $\CC[\ts x_1\lcd x_N\ts]$
can be defined as follows: 
the group $\Sym_N\subset \R_N$ acts by
permutations of the $N$ variables,
the element $x_p\in\R_N$ acts via mutiplication, 
and the element $y_p\in\R_N$ acts as the operator \eqref{dp}.
\end{prop}

In particular the operators \eqref{dp} with $p=1\lcd N$ 
pairwise commute. This fact is well known
and goes back to the celebrated work of Dunkl \cite{D}.
Due to this fact the proof of Proposition \ref{1.3} reduces to verifying
the second relation \eqref{sirel} and the 
relations \eqref{cross3},\eqref{cross4}
for the operator \eqref{dp} instead of $y_p\,$.
That is straightforward and we omit the~details. 

Observe that the algebra $\R_N$ contains a copy of the degenerate affine 
Hecke algebra $\H_N$ as a subalgebra. An injective homomorphism
$\H_N\to\R_N$ identical on the symmetric group $\Sym_N\subset\H_N$ can be 
defined by mapping $z_p\mapsto x_p\,y_p$ for each index $p=1\lcd N\ts$.
To prove the homomorphism property we need to verify the relations
\eqref{yrel},\eqref{ycom} for $x_p\,y_p\in\R_N$ instead of $z_p\in\H_N\,$.
The first of the two relations to be verified follows from 
\eqref{sirel}. Further, for~$p\neq q$
\begin{align*}
[\,x_p\,y_p\,,x_q\,y_q\,]&=
x_p\,(\ts x_q\,y_p-[\,x_q\,,y_p\,]\,)\,y_q-x_q\,y_q\,x_p\,y_p=
x_q\,x_p\,y_q\,y_p-x_p\,\si_{pq}\,y_q-x_q\,y_q\,x_p\,y_p
\\
&=x_q\,[\,x_p\,,y_q\,]\,y_p-x_p\,\si_{pq}\,y_q=
x_q\,\si_{pq}\,y_p-x_p\,\si_{pq}\,y_q=\si_{pq}\,(\ts x_p\,y_p-x_q\,y_q\ts)
\end{align*}
as needed. Here we used only the relations \eqref{sirel},\eqref{cross3}.
The injectivity of homomorphism $\H_N\to\R_N$ thus defined
follows from the bijectivity of the multiplication map \eqref{ratmult}.

Next we consider the \textit{trigonometric Cherednik algebra\/} $\T_N\,$.
It can be defined as the ring of fractions of the algebra $\R_N$
relative to the set of denominators $x_1\lcd x_N$. 
Multiplication in the algebra $\T_N$ provides a bijective linear map
\begin{equation*}
\label{trigmult}
\CC[\ts x_1\ts,x_1^{-1}\lcd x_N\ts,x_N^{-1}\ts]\ot
\CC\ts\Sym_N\ot\CC[\ts y_1\lcd y_N\ts]
\,\to\,\T_N\,;
\end{equation*}
see again \cite[Theorem 1.3]{EG}. Equivalently, the $\T_N$
can be defined as the complex associative algebra
generated by the ring 
$\CC[\ts x_1\ts,x_1^{-1}\lcd x_N\ts,x_N^{-1}\ts]$
of Laurent polynomials in $x_1\lcd x_N$ and by the
degenerate affine Hecke algebra $\H_N$ subject to 
natural relations
$
\si\,x_p\,\si^{-1}=x_{\ts\si(p)}
$
for all $\si\in\Sym_N$ and to the commutation relations
\begin{align}
\label{cross8}
[\,z_{p}\,,x_q\,]&=-\,x_p\,\si_{pq}
\quad\text{for}\quad
q\neq p\,;
\\[4pt]
\label{cross9}
[\,z_{p}\,,x_p\,]&=\ka\,x_p+\sum_{r\neq p}\,x_p\,\si_{pr}\,.
\end{align}
To prove the equivalence of two definitions of $\T_N$
we will use the embedding $\H_N\to\R_N$ constructed above.
Then we will only need to verify the relations
\eqref{cross8} and \eqref{cross9}
for the element 
$x_p\,y_p\in\R_N$ instead of 
$z_p\in\H_N\ts$.
This verification is direct by the defining
relations \eqref{cross3} and \eqref{cross4} in $\R_N\,$.
Using both definitions of $\T_N$ we obtain a corollary to
Proposition~\ref{1.3}.

\begin{cor}
\label{1.4}
An action of the algebra 
$\T_N$ on the vector space 
$\CC[\ts x_1\ts,x_1^{-1}\lcd x_N\ts,x_N^{-1}\ts]$
can be defined as follows: 
the elements $x_p\,,x_p^{-1}\in\T_N$ act via mutiplication,
the group $\Sym_N\subset\H_N$ acts by permutations of the $N$ variables, 
and the element $z_p\in\H_N$ acts as the operator
\begin{equation*}
\ka\,x_p\,\d_p+\sum_{r\neq p}\,\frac{x_p}{x_p-x_r}\,(1-\si_{pr})\,.
\end{equation*}
\end{cor}

Note that in the second definition of the algebra 
$\T_N$ we can also employ the pairwise commuting generators 
$u_1\lcd u_N\in\H_N$ 
instead of the non-commuting generators $z_1\lcd z_N\ts$; 
see the definition \eqref{zp}. Then instead of the defining relations 
\eqref{cross8},\eqref{cross9} in $\T_N$ we get
\begin{align}
\label{cross5}
[\,u_{p}\,,x_q\,]&=-\,x_q\,\si_{pq}
\quad\text{for}\quad
q<p\,;
\\[4pt]
\label{cross6}
[\,u_{p}\,,x_q\,]&=-\,x_p\,\si_{pq}
\quad\text{for}\quad
q>p\,;
\\[4pt]
\label{cross7}
[\,u_{p}\,,x_p\,]&=\ka\,x_p+
\sum_{r<p}\,x_r\,\si_{pr}+
\sum_{r>p}\,x_p\,\si_{pr}\,.
\end{align}
According to Corollary \ref{1.4} the element $u_p\in\H_N$ then acts on 
$\CC[\ts x_1\ts,x_1^{-1}\lcd x_N\ts,x_N^{-1}\ts]$ as 
\begin{equation*}
\ka\,x_p\,\d_p+\sum_{r\neq p}\,\frac{x_p}{x_p-x_r}\,(1-\si_{pr})
+\sum_{r<p}\,\si_{pr}\,.
\end{equation*}

Corollary \ref{1.4} is obtained by extending the action of 
$\R_N$ given by Proposition \ref{1.3} 
from the space $\CC[\ts x_1\lcd x_N\ts]$
to the space of Laurent polynomials 
$\CC[\ts x_1\ts,x_1^{-1}\lcd x_N\ts,x_N^{-1}\ts]\,$.
On the latter space
the element $y_p\in\R_N$ still acts as the Dunkl
operator \eqref{dp}. We will now describe a generalization
of Corollary \ref{1.4}, which goes back to the work of Cherednik \cite{C1}.


\subsection{}
\label{sec:13}

Consider the affine Lie algebra $\glhat_m\,$. By definition,
this is a central extension of~the \textit{current Lie algebra\/}
$\gl_m\ts[\,t,t^{-1}\ts]$ by a one-dimensional vector space with a
fixed basis element which we denote by $C\ts$.
Here $t$ is a formal variable.
Choose the basis of $\gl_m\ts[\,t,t^{-1}\ts]$ consisting of the elements
$E_{cd}\,t^{\,j}$ where $c,d=1\lcd m$ whereas 
$j$ ranges over~$\ZZ\ts$.
The commutators in the Lie algebra $\gl_m\ts[\,t,t^{-1}\ts]$ 
are taken pointwise so that
$$
[\,E_{ab}\,t^{\,i},E_{cd}\,t^{\,j}\,]=
(\ts\de_{bc}\ts E_{ad}-\de_{cb}\,E_{cd}\ts)\,t^{\,i+j}
$$
for the basis elements. In the extended Lie algebra $\glhat_m$ 
by definition we have the relations
\begin{equation}
\label{hatcom}
[\,E_{ab}\,t^{\,i},E_{cd}\,t^{\,j}\,]=
(\ts\de_{bc}\,E_{ad}-\de_{ad}\,E_{cd}\ts)\,t^{\,i+j}+
i\,\de_{\ts i,-j}\,\de_{ad}\,\de_{bc}\,C\ts.
\end{equation}

\newpage

Let $V$ be any module of the Lie algebra $\glhat_m$ such that
for any given vector in $V$ 
there is a degree $i$
such that the subspace $t^{\,i}\,\gl_m[t]\subset\glhat_m$ 
annihilates the vector.
Note that here the meaning of the symbol $V$
is different from that in Section \ref{sec:1},
where it was used to denote a module over the finite-dimensional 
Lie algebra $\gl_m\,$. Now consider the vector space
\begin{equation}
\label{w}
W=\CC[\ts x_1\ts,x_1^{-1}\lcd x_N\ts,x_N^{-1}\ts]\ot
(\CC^{\ts m})^{\ot N}\ot\ts V\ts.
\end{equation}
Due to our condition on $V$ for any $p=1\lcd N$ 
there is a well defined linear operator on $W$
\begin{equation}
\label{gap}
\sum_{i=0}^\infty\, 
\sum_{a,b=1}^m\,
x_p^{\,-i}\ot E_{ab}^{\ts(p)}\ot E_{\ts ba}\,t^{\,i}\,.
\end{equation}
Here $E_{ab}^{\ts(p)}$ is the operator \eqref{epab} acting on
$(\CC^{\ts m})^{\ot N}\ts$. Further, 
the symmetric group $\Sym_N$ acts on the tensor factor
$\CC[\ts x_1\ts,x_1^{-1}\lcd x_N\ts,x_N^{-1}\ts]$ of $W$
by permutations of the $N$ variables.
There is another copy of the group $\Sym_N$ acting on 
the $N$ tensor factors $\CC^{\ts m}$ of $W$ by their permutations.
Using these two actions of~$\Sym_N$
for any $p=1\lcd N$ introduce the \textit{Cherednik operator} on $W$
\begin{gather}
\nonumber
\ka\,x_p\,\d_p\ot\id^{\,\ot N}\ot\id\,+\,
\sum_{r\neq p}\,\frac{x_p}{x_p-x_r}\,(1-\si_{pr})\ot\si_{pr}\ot\id\ +
\\
\label{up}
\sum_{i=0}^\infty\,
\sum_{a,b=1}^m\, 
x_p^{\,-i}\ot E_{ab}^{\ts(p)}\ot E_{\ts ba}\,t^{\,i}\,.
\end{gather}

The vector space
\begin{equation}
\label{u}
\CC[\ts x_1\ts,x_1^{-1}\lcd x_N\ts,x_N^{-1}\ts]\ot
(\CC^{\ts m})^{\ot N}
\end{equation}
can be naturally identified with the tensor product of $N$ copies
of the space $\CC^{\ts m}[\ts t\ts,t^{-1}\ts]\,$.
The latter space can be regarded as a $\glhat_m\ts$-module
where the central element $C$ acts as zero.
By taking the tensor product of $N$ copies of this module
with $V$ we turn the vector space 
$W$ to a $\glhat_m\ts$-module. 
The element $E_{cd}\,t^{\,j}$ of $\glhat_m$ acts on $W$ as
\begin{equation}
\label{ecdj}
\sum_{q=1}^N\,x_q^{\,j}\ot E_{cd}^{\ts(q)}\ot\id+
\id\ot\id^{\,\ot N}\ot E_{\ts cd}\,t^{\,j}\,.
\end{equation}

\vspace{-4pt}

We will also employ the affine Lie algebra $\slhat_m\,$.  
This is a subalgebra of $\glhat_m$ spanned by
the subspace $\sl_m\ts[\,t,t^{-1}\ts]\subset\gl_m\ts[\,t,t^{-1}\ts]$
and by the central elements $C\ts$. 
For any scalar $\ell\in\CC\,$, a module of the Lie algebra
$\glhat_m$ or $\slhat_m$ is said to be of \emph{level\/} $\ell$ 
if $C$ acts on this module as that scalar.
In particular, the $\glhat_m\ts$-module
$\CC^{\ts m}[\ts t\ts,t^{-1}\ts]$ used 
above was of level zero.
Note that by the defining 
relations \eqref{hatcom},
the subspaces
$\gl_m\subset\gl_m\ts[\,t,t^{-1}\ts]$ and
$\sl_m\subset\sl_m\ts[\,t,t^{-1}\ts]$ are
Lie subalgebras of $\glhat_m$ and $\slhat_m$ respectively.
Denote by $\q$ the subspace
$t^{-1}\,\sl_m\ts[\,t^{-1}\ts]\subset\sl_m\ts[\,t,t^{-1}\ts]\,$.  
This is a subalgebra of both $\slhat_m$ and $\glhat_m\,$.
Note that here the meaning of the symbol $\q$
is different from that in Section \ref{sec:1},
where it denoted some subalgebra of $\gl_m\,$.
We can now state the main properties
of Cherednik operators on $W$ due to \cite{AST,S2}.

\begin{prop}
\label{1.5}
{\rm(i)}
Using the\/ $\glhat_m\ts$-module structure on $V$, 
an action of the algebra\/ $\T_N$ on the vector space \eqref{w}
is defined as follows: 
the elements $x_p\,,x_p^{-1}\in\T_N$ act via mutiplication in\/
$\CC[\ts x_1\ts,x_1^{-1}\lcd x_N\ts,x_N^{-1}\ts]\,,$
the group $\Sym_N\subset\H_N$ acts 
by simultaneous permutations of the variables $x_1\lcd x_N$ and of the $N$ 
tensor factors\/ $\CC^{\ts m}$,
and the element $z_p\in\H_N$ acts as~\eqref{up}. 
\\
{\rm(ii)} 
This action of\/ $\T_N$ on $W$ commutes with the action of 
the Lie subalgebra\/ $\gl_m\subset\glhat_m\,$.
\\[1pt]
{\rm(iii)}
If\/ $V$ has level\/ $\ka-m$ then the action of\/ $\T_N$
preserves the subspace\, 
$\q\,W\subset W$.
\end{prop}

In the next section we will give a proof of this proposition
since many details have been omitted in \cite{AST,S2}. 
By inspecting that proof we will also get a version
of Proposition \ref{1.5} for the vector space \eqref{w} where
the tensor factor $V$ is a module not of $\glhat_m$ but
only of $\slhat_m\,$. 
In the latter case, for any vector in $V$
we still assume the existence of 
a degree $i$ such that the subspace $t^{\,i}\,\sl_m[t]\subset\slhat_m$
annihilates the vector.
Due to \eqref{i}, an action of the sum
\begin{equation}
\label{gaps}
\sum_{i=0}^\infty\,
x_p^{\,-i}\ot
\Bigl(\,
\sum_{a,b=1}^m\,
E_{ab}^{\ts(p)}\ot E_{\ts ba}\,t^{\,i}-
\frac1m\ \id^{\,\ot\ts N}\ot I\,t^{\,i}
\,\Bigr)
\end{equation}
can be then defined on the vector space \eqref{w} by using only the
$\slhat_m\ts$-module structure of $V$. Then for every $p=1\lcd N$ we 
get a modification of the Cherednik operator \eqref{up}
on $W$,
\begin{gather}
\nonumber
\ka\,x_p\,\d_p\ot\id^{\,\ot N}\ot\id\,+\,
\sum_{r\neq p}\,\frac{x_p}{x_p-x_r}\,(1-\si_{pr})\ot\si_{pr}\ot\id\ +
\\
\label{ups}
\sum_{i=0}^\infty\,
x_p^{\,-i}\ot
\Bigl(\,
\sum_{a,b=1}^m\,
E_{ab}^{\ts(p)}\ot E_{\ts ba}\,t^{\,i}-
\frac1m\ \id^{\,\ot\ts N}\ot I\,t^{\,i}
\,\Bigr)
\,.
\end{gather}
Here we use the sum \eqref{gaps} instead of 
the sum \eqref{gap} used in 
\eqref{up}. Further, we can turn the vector space
\eqref{w} into another $\slhat_m\ts$-module by regarding \eqref{u}
as $\slhat_m\ts$-module of level~zero. 

\begin{cor}
\label{1.6}
{\rm(i)}
Using the\/ $\slhat_m\ts$-module structure on $V$, 
an action of the algebra\/ $\T_N$ on the vector space \eqref{w}
can be defined as follows: 
the elements $x_p\,,x_p^{-1}\in\T_N$ act via mutiplication in\/
$\CC[\ts x_1\ts,x_1^{-1}\lcd x_N\ts,x_N^{-1}\ts]\,,$
the group $\Sym_N\subset\H_N$ acts 
by simultaneous permutations of the variables $x_1\lcd x_N$ and of the $N$ 
tensor factors\/ $\CC^{\ts m}$,
and the element $z_p\in\H_N$ acts as~\eqref{ups}. 
\\
{\rm(ii)} 
This action of\/ $\T_N$ on $W$ commutes with the action of 
the Lie subalgebra\/ $\sl_m\subset\slhat_m\,$.
\\[1pt]
{\rm(iii)}
If\/ $V$ has level\/ $\ka-m$ then the action of\/ $\T_N$
preserves the subspace\, 
$\q\,W\subset W$.
\end{cor}


\subsection{}
\label{sec:14}

By using Corollary \ref{1.6}(i) and the definition \eqref{w} 
we get a functor
$
\A_N: V\mapsto W
$
from the category of all $\slhat_m\ts$-modules 
satisfying the annihilation condition stated just before 
\eqref{gaps}. Note that the resulting 
actions of $\slhat_m$ and $\T_N$ on $W$ do not commute in general.
However, this will be our analogue for $\slhat_m$ of the functor $\F_N$
introduced in the end of Subsection~\ref{sec:11}.
Let us now relate the two functors to each other.

For any given $\sl_m\ts$-module $U$ and
$\ell\in\CC$ let $V$ be the $\slhat_m\ts$-module
of level $\ell$ {parabolically induced} from $U$. 
Note that the subspace of $\slhat_m$
spanned by $\sl_m[t]$ and by the central element $C$ is a Lie subalgebra.
Denote this subalgebra by $\p\,$.
To define $V$ we first extend the action on $U$ from $\sl_m$ to $\p\,$
so that all the elements of
$t\,\sl_m[t]$ act as zero, while $C$ acts as the scalar
$\ell\,$. By definition, $V$ is the $\slhat_m\ts$-module
induced from the $\p\ts$-module~$U$. 
Since the element $C$ is central in $\slhat_m\,$, 
the module $V$ indeed has level $\ell\,$. 
It also satisfies the annihilation condition 
mentioned in the previous paragraph.

Let us apply the functor $\A_N$ to this $V$. 
Using the notation $\q=t^{-1}\,\sl_m\ts[\,t^{-1}\ts]$ from 
the previous subsection, we have a vector space decomposition
$\slhat_m=\p\op\q\,$.
Note again that here the meaning of the symbols $\p$ and $\q$
is different from that in Section \ref{sec:1} but similar. 
Consider the space $\A_N(V)_{\ts\q}$ of {coinvariants} of
$\A_N(V)$ relative to the action of the subalgebra $\q\subset\slhat_m\,$.
This space is the quotient of the vector space
$\A_N(V)$ by the image of the action of~$\q\,$.
The subalgebra $\sl_m\subset\slhat_m$ acts on this quotient,
because the adjoint action of this subalgebra on $\slhat_m$ preserves $\q\,$. 
If $\,\ell=\ka-m$ then due to Corollary \ref{1.6}(iii)
the algebra $\T_N$ also acts on this quotient. Moreover, 
by Corollary \ref{1.6}(ii) 
the latter action commutes with the action of $\sl_m\,$. 
Thus the space of $\q\ts$-coinvariants $\A_N(V)_{\ts\q}$ becomes a bimodule 
over $\sl_m$ and $\T_N\,$.

On the other hand, we can apply the functor $\F_N$ 
to the $\sl_m\ts$-module $U$ as given above. 
Hence we obtain a bimodule $\F_N(U)$ of $\sl_m$ and of the 
degenerate affine Hecke algebra $\H_N\ts$, see Corollary \ref{1.2}(i,ii).
Since $\H_N$ is a subalgebra of 
$\T_N\ts$, we have the induction functor
$\Ind_{\,\H_N}^{\,\T_N}\,$.
If we apply it to any bimodule
over $\sl_m$ and\/ $\H_N\,$, we will get a bimodule 
over $\sl_m$ and\/ $\T_N\,$.  

\begin{theorem}
\label{1.7}
Let\/ $\ell=\ka-m\,$. \!Then the bimodule $\A_N(V)_{\ts\q}$ of\/
$\sl_m$ and\/ $\T_N$ is equivalent~to\/
\begin{equation}
\label{indf}
\Ind_{\,\H_N}^{\,\T_N}\F_N(U)\,.
\end{equation}
\end{theorem}

\begin{proof}
Using the second definition of the algebra $\T_N$
and the definition of the functor $\F_N\ts$, the vector space of the 
bimodule \eqref{indf} can be identified with the tensor product  
\begin{equation}
\label{indx}
\CC[\ts x_1\ts,x_1^{-1}\lcd x_N\ts,x_N^{-1}\ts]\ot
(\CC^{\ts m})^{\ot N}\ot U\,.
\end{equation}
The subalgebra 
$\CC[\ts x_1\ts,x_1^{-1}\lcd x_N\ts,x_N^{-1}\ts]\subset\T_N$
acts on \eqref{indx} via
left multiplication in the first tensor factor. Further,
the subalgebra $\H_N\subset\T_N$ acts on \eqref{indx}
preserving the subspace 
$1\ot (\CC^{\ts m})^{\ot N}\ot U$.
The action of $\H_N$ on this subspace is determined by naturally 
identifying it with $(\CC^{\ts m})^{\ot N}\ot U$, see Corollary \ref{1.2}(i).
All this defines the action of $\T_N$ on \eqref{indx}. Note that
$\sl_m$ acts diagonally on the $N$ tensor factors 
$\CC^{\ts m}$ and on the last tensor factor $U$ of~\eqref{indx}. 

On the other hand, the 
vector space of the $\slhat_m\ts$-module $V$
can be identified with 
$\U(\q)\ot U$ whereon the Lie subalgebra $\q\subset\slhat_m$ acts via
left multiplication in the first tensor factor.
The vector space of the given $\sl_m\ts$-module $U$
gets identified with the subspace 
\begin{equation}
\label{subu}
1\ot U\subset \U(\q)\ot U\ts.
\end{equation}
Then by the definition of a parabolically induced module of level $\ell\ts$, 
on this subspace\ts: 
any element of the subalgebra $t\,\sl_m[t]\subset\slhat_m$ 
acts as zero, the subalgebra
$\sl_m\subset\slhat_m$ acts via its defining action on $U$,
the element $C$ of $\slhat_m$ acts
as the the scalar $\ell\ts$. 
All this determines the action of the Lie algebra
$\slhat_m$ on our $V=\U(\q)\ot U$.

Now consider the module $W=\A_N(V)$ over $\slhat_m$ and $\T_N\,$.
Using the above identification,
\begin{equation}
\label{uu}
W=\CC[\ts x_1\ts,x_1^{-1}\lcd x_N\ts,x_N^{-1}\ts]\ot
(\CC^{\ts m})^{\ot N}\ot\ts \U(\q)\ot U
\end{equation}
as a vector space. 
The corresponding space of $\q\ts$-coinvariants is isomorphic to \eqref{indx}.
Indeed, we can define a bijective linear mapping $\io$ from 
\eqref{indx} to the quotient vector space $W\ts/\ts\q\,W$ as follows.
First we map \eqref{indx} to the subspace
$$
\CC[\ts x_1\ts,x_1^{-1}\lcd x_N\ts,x_N^{-1}\ts]\ot
(\CC^{\ts m})^{\ot N}\ot\ts 1\ot U\subset W
$$
in the natural way, and then regard the image of the latter map 
modulo $\q\,W$. To prove the bijectivity of the resulting mapping 
$\io\,$, consider the ascending $\ZZ\ts$-filtration on the 
space
\eqref{u} defined by the total degree in the variables $x_1\lcd x_N\,$. 
The action of the subalgebra
$\q\subset\slhat_m$ preserves 
the filtration. 
Moreover, the corresponding graded action of $\q$ is trivial.
Hence this filtration on \eqref{u}
induces an ascending filtration on $W$ such that the
corresponding graded action of $\q$ is via left multiplication in the
tensor factor $\U(\q)$ of \eqref{uu}.
The bijectivity of the mapping 
$\io$ now follows from 
the Poincar\'e-Birkhoff-Witt theorem for the 
algebra $\U(\q)\ts$.

Since the subalgebra $\sl_m\subset\slhat_m$ acts on the subspace \eqref{subu}
through its defining action on $U$, the mapping $\io$ is 
$\sl_m\ts$-equivariant. Further, the subalgebra
$\CC[\ts x_1\ts,x_1^{-1}\lcd x_N\ts,x_N^{-1}\ts]\subset\T_N$
acts on both \eqref{indx} and $W$ via left multiplication
in itself as their tensor factor. Hence
the mapping $\io$ is equivariant for this subalgebra. 
But relative to the action of this subalgebra on the vector space
\eqref{indx}, the subspace 
$1\ot (\CC^{\ts m})^{\ot N}\ot U$ is cyclic.
To complete the proof of Theorem \ref{1.7} 
it now remains to check
the $\H_N\ts$-equivariance of the restriction of the mapping
$\io$ to this subspace of \eqref{indx}.
Here $\H_N$ is regarded as a subalgebra of $\T_N\,$.

The image of the subspace $1\ot (\CC^{\ts m})^{\ot N}\ot U$ 
of \eqref{indx} under $\io$ is the subspace
\begin{equation}
\label{11u}
1\ot(\CC^{\ts m})^{\ot N}\ot\ts 1\ot U\subset W
\end{equation}
regarded modulo $\q\,W$. The action of the symmetric group 
$\Sym_N\subset\H_N$ on both the subspace of \eqref{indx}
and its image under $\io$ 
is via permutations of the $N$ tensor factors $\CC^{\ts m}\ts$.
This action of $\Sym_N$ obviously commutes with $\io\,$.
Now consider the elements $z_1\lcd z_N\in\H_N$
which together with $\Sym_N$ generate the algebra $\H_N$. 
According to Corollary \ref{1.2}(i) for any $p=1\lcd N$
the generator $z_p$ acts on
the subspace $1\ot (\CC^{\ts m})^{\ot N}\ot U$ of \eqref{indx} as 
$$
\sum_{a,b=1}^m 
\id\ot E_{ab}^{\ts(p)}\ot E_{\ts ba}-\frac1m\ \id\ot\id^{\,\ot\ts N}\ot I\,.
$$

On the other hand, according to Corollary \ref{1.6}(i) 
the generator $z_p$ acts on $W$ as \eqref{ups}.
When applied to elements of the subspace \eqref{11u},
all summands displayed in the first line of \eqref{ups} vanish,
because here the derivation $\d_p$ and the difference $1-\si_{pr}$
are applied to the constant function
$1\in\CC[\ts x_1\ts,x_1^{-1}\lcd x_N\ts,x_N^{-1}\ts]\,$.
The summands in the second line of 
\eqref{ups} corresponding to the indices 
$i=1,2,\ldots$ also vanish on the subspace \eqref{11u},
because the last tensor factor of any of these summands
belongs to the subalgebra $t\,\sl_m[t]\subset\slhat_m\,$. 
The remaining summands of \eqref{ups} correspond to 
$i=0\,$. Their sum acts on the subspace~\eqref{11u}~as 
$$
\sum_{a,b=1}^m 
\id\ot E_{ab}^{\ts(p)}\ot\id\ot E
_{\ts ba}-\frac1m\ \id\ot\id^{\,\ot\ts N}\ot\id\ot I\,.
$$
Hence the restriction of 
$\io$ 
to the subspace $1\ot (\CC^{\ts m})^{\ot N}\ot U$ 
of \eqref{indx} is $z_p\ts$-equivariant.  
\end{proof}


\section{Proof of Proposition \ref{1.5}}
\medskip


\subsection{}
\label{sec:31}

To prove the part (i) of Proposition \ref{1.5}, we will use the
definition of the algebra $\T_N$ as the ring of fractions of the algebra
$\R_N$ relative to the set of denominators $x_1\lcd x_N\ts$. 
Let us denote by $X_p$ the operator
of multiplication by the variable $x_p$ in the first tensor factor of
the vector space \eqref{w}. Let us further denote by $Y_p$ the operator
\begin{gather*}
\ka\,\d_p\ot\id^{\,\ot N}\ot\id\,+\,
\sum_{r\neq p}\,\frac{1}{x_p-x_r}\,(1-\si_{pr})\ot\si_{pr}\ot\id\,+\,
\sum_{i=0}^\infty\,
\sum_{a,b=1}^m\, 
x_p^{\,-i-1}\ot E_{ab}^{\ts(p)}\ot E_{\ts ba}\,t^{\,i}
\end{gather*}
on \eqref{w}, so that  
the Cherednik operator \eqref{up} equals the composition $X_p\,Y_p\,$.
We will show that a representation of the algebra $\R_N$ on the 
vector space \eqref{w} can be defined by mapping
\begin{equation}
\label{xyrep}
x_p\mapsto X_p\ts,
\quad
y_p\mapsto Y_p
\quad\text{and}\quad
\si\mapsto\si\ot\si\ot\id
\end{equation}
for all $p=1\lcd N$ and $\si\in\Sym_N\ts$. We will use 
the defining relations \eqref{sirel},\eqref{cross3},\eqref{cross4}
of $\R_N\ts$.

The definitions of the operators $X_p$ and $Y_p$ immediately show
that the relations \eqref{sirel} are satisfied under the mapping
\eqref{xyrep}. Moreover, due to the latter fact it suffices to consider 
the relations \eqref{cross3},\eqref{cross4} only for $p=1\ts$.
In this case the commutator $[\,Y_p\,,X_q\,]$ for $q>1$ equals
$$
\sum_{r>1}\,\frac{1}{x_1-x_r}\,[\,1-\si_{1r}\,,x_q\,]\ot\si_{1r}\ot\id
\,=\,
\frac{1}{x_1-x_q}\,[\,-\,\si_{1q}\,,x_q\,]\ot\si_{1q}\ot\id
\,=\,
-\,\si_{1q}\ot\si_{1q}\ot\id
$$
as required by the relation \eqref{cross3}. Further, in the case $p=1$
the commutator $[\,Y_p\,,X_p\,]$ equals
$$
\ka+\sum_{r>1}\,\frac{1}{x_1-x_r}\,[\,1-\si_{1r}\,,x_1\,]\ot\si_{1r}\ot\id
\,=\,
\ka+\sum_{r>1}\,\si_{1r}\ot\si_{1r}\ot\id
$$
as required by the relation 
\eqref{cross4}. To complete the proof of the part (i) of 
Proposition \ref{1.5}, it now remains to check
the pairwise commutativity of the operators $Y_1\lcd Y_N\ts$. 

Extend the vector space \eqref{w} 
by replacing its first tensor factor
$\CC[\ts x_1\ts,x_1^{-1}\lcd x_N\ts,x_N^{-1}\ts]$ 
by the space of
all complex valued 
rational functions in 
$x_1\lcd x_N$ with the permutational action of $\Sym_N\ts$.
For any $p$ consider the following three
operators on the extended vector space,
\begin{gather*}
D_p\,=\,\d_p\ot\id^{\,\ot N}\ot\id\ts,
\quad
R_{\ts p}\,=\,
\sum_{r\neq p}\,\frac{1}{x_p-x_r}\,\si_{pr}\ot\si_{pr}\ot\id\ts,
\\
T_p\,=\,
\sum_{r\neq p}\,\frac{1}{x_p-x_r}\ot\si_{pr}\ot\id\,+\,
\sum_{i=0}^\infty\,
\sum_{a,b=1}^m\, 
x_p^{\,-i-1}\ot E_{ab}^{\ts(p)}\ot E_{\ts ba}\,t^{\,i}
\end{gather*}
so that the operator $Y_p$ is the restriction of the operator
$\ka\,D_p-R_{\ts p}+T_p$ to the space~\eqref{w}.
The operators $D_1\lcd D_N$ obviously pairwise commute.
The next two lemmas show that~the operators $R_{\ts 1}\lcd R_{\ts N}$ and
$T_1\lcd T_N$ enjoy the same property.

\begin{lemma}
\label{3.1}
The operators $R_{\ts 1}\lcd R_{\ts N}$ pairwise commute.
\end{lemma}

\begin{proof}
It suffices to prove the commutativity of 
the operators $R_{\ts p}$ and $R_{\ts q}$ only for $p=1$ and $q=2$. 
By definition, the commutator $[\,R_{\ts 1}\,,R_{\ts 2}\,]$ is
equal to the sum
$$
\sum_{\substack{r\neq 1\\s\neq 2}}\ 
\Bigl[\ 
\frac{1}{x_1-x_r}\,\si_{1r}\ot\si_{1r}\ot\id
\ ,\,
\frac{1}{x_2-x_s}\,\si_{2s}\ot\si_{2s}\ot\id
\ \Bigr]
\vspace{-4pt}
$$
which is in turn equal to the sum over the indices $r>2$ of  
\begin{align*}
\Bigl[\ 
\frac{1}{x_1-x_2}\,\si_{12}\ot\si_{12}\ot\id
&\ ,\,
\frac{1}{x_1-x_r}\,\si_{1r}\ot\si_{1r}\ot\id
\ \Bigr]\,+
\\
\Bigl[\ 
\frac{1}{x_1-x_2}\,\si_{12}\ot\si_{12}\ot\id
&\ ,\,
\frac{1}{x_2-x_r}\,\si_{2r}\ot\si_{2r}\ot\id
\ \Bigr]\,+
\\
\Bigl[\  
\frac{1}{x_1-x_r}\,\si_{1r}\ot\si_{1r}\ot\id
&\ ,\,
\frac{1}{x_2-x_r}\,\si_{2r}\ot\si_{2r}\ot\id
\ \Bigr]\,.
\end{align*}
For every single index $r>2\ts$, the sum of the three commutators displayed 
above is equal to zero. Indeed, because the action of
the group $\Sym_N$ on the space of complex valued rational 
functions in $x_1\lcd x_N$ by permutations of the variables is faithful,
it suffices to prove that 
$$
\Bigl[\  
\frac{1}{x_1-x_2}\,\si_{12}
\,,
\frac{1}{x_1-x_r}\,\si_{1r}
\,\Bigr]+
\Bigl[\  
\frac{1}{x_1-x_2}\,\si_{12}
\,,
\frac{1}{x_2-x_r}\,\si_{2r}
\,\Bigr]+
\Bigl[\  
\frac{1}{x_1-x_r}\,\si_{1r}
\,,
\frac{1}{x_2-x_r}\,\si_{2r}
\,\Bigr]=0\,.
$$
The last equality can be easily
verified by direct calculation.
\end{proof}

\vspace{-14pt}

\begin{lemma}
\label{3.2}
The operators $T_1\lcd T_N$ pairwise commute.
\end{lemma}

\begin{proof}
It suffices to prove the commutativity of the operators 
$T_p$ and $T_q$ only for $p=1$ and $q=2$. 
The commutator $[\,T_1\,,T_2\,]$ is
equal to
\begin{gather*}
\Bigl[\  
\frac{1}{x_1-x_2}\ot\si_{12}\ot\id
\ ,\,
\sum_{i=0}^\infty\,
\sum_{a,b=1}^m\, 
x_1^{\,-i-1}\ot E_{ab}^{\ts(1)}\ot E_{\ts ba}\,t^{\,i}
\,\Bigr]\,+
\\
\Bigl[\  
\frac{1}{x_1-x_2}\ot\si_{12}\ot\id
\ ,\,
\sum_{j=0}^\infty\,
\sum_{c,d=1}^m\, 
x_2^{\,-j-1}\ot E_{cd}^{\ts(2)}\ot E_{\ts dc}\,t^{\,j}
\,\Bigr]\,+
\\
\Bigl[\ 
\sum_{i=0}^\infty\,
\sum_{a,b=1}^m\, 
x_1^{\,-i-1}\ot E_{ab}^{\ts(1)}\ot E_{\ts ba}\,t^{\,i}
\ ,\,
\sum_{j=0}^\infty\,
\sum_{c,d=1}^m\, 
x_2^{\,-j-1}\ot E_{cd}^{\ts(2)}\ot E_{\ts dc}\,t^{\,j}
\,\Bigr]
\end{gather*}
plus the sum over the indices $r>2$ of  
\begin{align*}
&\Bigl[\ 
\frac{1}{x_1-x_2}\ot\si_{12}\ot\id
\ ,\,
\frac{1}{x_1-x_r}\ot\si_{1r}\ot\id
\ \Bigr]\,+
\\
&\Bigl[\ 
\frac{1}{x_1-x_2}\ot\si_{12}\ot\id
\ ,\,
\frac{1}{x_2-x_r}\ot\si_{2r}\ot\id
\ \Bigr]\,+
\\
&\Bigl[\  
\frac{1}{x_1-x_r}\ot\si_{1r}\ot\id
\ ,\,
\frac{1}{x_2-x_r}\ot\si_{2r}\ot\id
\ \Bigr]\,.
\end{align*}
Here we have omitted the zero commutators,
see the beginning of the proof of Lemma \ref{3.1}.

For any $r>2\ts$, the sum of the last three 
displayed commutators equals zero. The sum of the three
commutators in the previous display also equals zero.
Both equalities follow from the {\it classical Yang\ts-Baxter equation}
for the rational function of two complex variables~$u$~and~$v$
\begin{equation}
\label{r}
\frac1{u-v}\ 
\sum_{a,b=1}^m\,E_{ab}\ot E_{\ts ba}
\end{equation}
with values in $\gl_m\ot\gl_m\,$, see for instance \cite[Section 3.2]{AST}. 
To derive the first stated equality, 
we observe that the sum over $a,b=1\lcd m$ in \eqref{r}
acts on $\CC^m\ot\CC^m$ as the permutation of the two tensor factors.
To derive the second stated equality, we use the expansion
$$
\frac1{u-v}\,=\,
\sum_{i=0}^\infty\,u^{\,-i-1}\,v^{\,i}
$$
and also observe that for $i\ts,j\ge0$ the summand at the right hand side 
of \eqref{hatcom} involving the central element $C\in\glhat_m$ vanishes.
\end{proof}

Part (i) of Proposition \ref{1.5} follows from the two lemmas above, 
and from the next three.

\begin{lemma}
\label{3.3}
For any $p\neq q$ we have the equality\/ 
$[\,D_p\,,R_{\ts q}\,]+[\,R_{\ts p}\,,D_q\,]=0\ts$.
\end{lemma}

\begin{proof}
It suffices to prove the stated equality
only for $p=1$ and $q=2$. By definition,
\begin{gather*}
[\,D_1\,,R_{\ts 2}\,]\,=
\sum_{\substack{r\neq 2}}\ 
\Bigl[\ 
\d_1\ot\id^{\,\ot N}\ot\id
\ ,\,
\frac{1}{x_2-x_r}\,\si_{2r}\ot\si_{2r}\ot\id
\ \Bigr]\,=
\vspace{-4pt}
\\
\frac{1}{(\ts x_1-x_2\ts)^2\ts}\,\si_{12}\ot\si_{12}\ot\id
+
\frac{1}{x_1-x_2}\,\si_{12}\,(\ts\d_1-\d_2\ts)\ot\si_{12}\ot\id\,.
\end{gather*}
The sum in the last displayed line is invariant under exchanging
the indices $1$~and~$2\ts$. Hence the commutator $[\ts D_2\,,R_{\ts 1}\ts]$
is equal to the same sum. Therefore
$[\,D_1\,,R_{\ts 2}\,]+[\,R_{\ts 1}\,,D_2\,]=0\ts$.
\end{proof}

\vspace{-14pt}

\begin{lemma}
\label{3.4}
For any $p\neq q$ we have the equality\/ 
$[\,D_p\,,T_q\,]+[\,T_p\,,D_q\,]=0\ts$.
\end{lemma}

\begin{proof}
It suffices to prove the stated equality
only for $p=1$ and $q=2$. By omitting the~zero commutators, we get
$$
[\,D_1\,,T_2\,]=
\Bigl[\,\d_1\ts,\ts\frac{1}{x_2-x_1}
\,\Bigr]\ot\si_{12}\ot\id
=
\frac{1}{(\ts x_1-x_2\ts)^2\ts}\ot\si_{12}\ot\id
$$
which is again invariant under exchanging
the indices $1$~and~$2\ts$. Therefore the commutator
$[\,D_2\,,T_1\,]$ is the same as $[\,D_1\,,T_2\,]\ts$.
Thus we get the equality
$[\,D_1\,,T_2\,]+[\,T_1\,,D_2\,]=0\ts$.
\end{proof}

\vspace{-14pt}

\begin{lemma}
\label{3.5}
For any $p\neq q$ we have the equality\/ 
$[\,R_{\ts p}\,,T_q\,]+[\,T_p\,,R_{\ts q}\,]=0\ts$.
\end{lemma}

\begin{proof}
It suffices to prove the stated equality only
for $p=1$ and $q=2$. 
By omitting the zero commutators like we did in 
our proof of Lemma \ref{3.2}, the commutator
$[\,R_{\ts 1}\,,T_2\,]$ is~equal~to
\begin{equation}
\label{3.51}
\Bigl[\ 
\frac{1}{x_1-x_2}\,\si_{12}\ot\si_{12}\ot\id
\ ,\,
\sum_{i=0}^\infty\,
\sum_{a,b=1}^m\, 
x_2^{\,-i-1}\ot E_{ab}^{\ts(2)}\ot E_{\ts ba}\,t^{\,i}
\,\Bigr]
\end{equation}
plus the sum over the indices $r>2$ of
\begin{align}
\notag
-\,\Bigl[\    
\frac{1}{x_1-x_r}\,\si_{1r}\ot\si_{1r}\ot\id
&\ ,\,
\frac{1}{x_1-x_2}\ot\si_{12}\ot\id
\,\Bigr]\,+
\\
\notag
\Bigl[\ 
\frac{1}{x_1-x_r}\,\si_{1r}\ot\si_{1r}\ot\id
&\ ,\,
\frac{1}{x_2-x_r}\ot\si_{2r}\ot\id
\,\Bigr]\,+
\\
\label{3.52}
\Bigl[\  
\frac{1}{x_1-x_2}\,\si_{12}\ot\si_{12}\ot\id
&\ ,\,
\frac{1}{x_2-x_r}\ot\si_{2r}\ot\id
\,\Bigr]\,.
\end{align}
In particular, here we used the vanishing of the commutator
$$
\ \Bigl[\ 
\frac{1}{x_1-x_2}\,\si_{12}\ot\si_{12}\ot\id
\ ,\,
\frac{1}{x_1-x_2}\ot\si_{12}\ot\id
\,\Bigr]\,.
$$

The commutator displayed in the single line \eqref{3.51} can be rewritten 
as the sum
$$
\sum_{i=0}^\infty\,
\sum_{a,b=1}^m\, 
\Bigl(\ 
\frac{x_1^{\,-i-1}}{x_1-x_2}\,\si_{12}\ot
E_{ab}^{\ts(1)}\si_{12}\ot E_{\ts ba}\,t^{\,i}
\,-\,
\frac{x_2^{\,-i-1}}{x_1-x_2}\,\si_{12}\ot
E_{ab}^{\ts(2)}\si_{12}\ot E_{\ts ba}\,t^{\,i}
\,\Bigr)
$$
which is obviously invariant under exchanging the indices $1$ and $2\ts$.
For any index $r>2\ts$, the first
commutator displayed in the three lines \eqref{3.52} 
can be rewritten as
$$
-\ \frac{1}{(x_1-x_r)(x_2-x_r)}\,\si_{1r}\ot\si_{1r}\,\si_{12}\ot\id
\,-\,
\frac{1}{(x_1-x_2)(x_1-x_r)}\,\si_{1r}\ot\si_{12}\,\si_{1r}\ot\id\,,
$$
while the second commutator in \eqref{3.52} can be rewritten as
$$
-\ \frac{1}{(x_1-x_2)(x_1-x_r)}\,\si_{1r}\ot\si_{1r}\,\si_{2r}\ot\id
\,-\,
\frac{1}{(x_1-x_r)(x_2-x_r)}\,\si_{1r}\ot\si_{2r}\,\si_{1r}\ot\id\,.
$$
Hence these two commutators cancel each other in \eqref{3.52}
by the relations $\si_{1r}\,\si_{12}=\si_{2r}\,\si_{1r}$ and
$\si_{12}\,\si_{1r}=\si_{1r}\,\si_{2r}\,$.
The third commutator in \eqref{3.52} can be rewritten as the difference
$$
\frac{1}{(x_1-x_2)(x_1-x_r)}\,\si_{12}\ot\si_{12}\,\si_{2r}\ot\id
\,-\,
\frac{1}{(x_1-x_2)(x_2-x_r)}\,\si_{12}\ot\si_{2r}\,\si_{12}\ot\id
$$
which is invariant under exchanging the indices $1$ and $2$
due to the relations $\si_{12}\,\si_{1r}=\si_{2r}\,\si_{12}$ and
$\si_{1r}\,\si_{12}=\si_{12}\,\si_{2r}\ts$.
Therefore the commutator
$[\,R_{\ts 2}\,,T_1\,]$ is the same as the $[\,R_{\ts 1}\,,T_2\,]$
which we computed above.
Thus we get the required equality
$[\,R_{\ts 1}\,,T_2\,]+[\,T_1\,,R_{\ts 2}\,]=0\ts$. 
\end{proof}

Thus we have proved part (i) of Proposition \ref{1.5}. Moreover,
the above five lemmas imply that for any given $\ep\in\CC$ the operators
$\ka\,D_p+\ep\,R_{\ts p}+T_p$ with $p=1\lcd N$ pairwise commute.
However, the choice $\ep=-1$ is necessary for these operators
to preserve the vector space $W$.
The defining relations 
\eqref{cross3} and \eqref{cross4} of the algebra $\R_N$
exhibit this particular choice of $\ep\,$.
  
By inspecting the arguments given in this subsection,
we also get part (i) of Corollary~\ref{1.6}.
Indeed, if we replace the operator \eqref{up} by \eqref{ups}
then the definitions of $Y_p$ and $T_p$ given in the beginning 
in this subsection have to be modified by substracting 
\begin{equation}
\label{subtr}
\frac1m\,\ts\sum_{i=0}^\infty\,
x_p^{\,-i-1}\ot\id^{\,\ot\ts N}\ot I\,t^{\,i}\,.
\end{equation}
With these modifications only, all other equalities stated in 
this subsection will remain valid. 
Note that in the modified arguments we still regard $V$
as a module over the Lie algebra~$\glhat_m\,$.


\subsection{}
\label{sec:32}

In this subsection we will prove the 
parts (ii) and (iii) of Proposition \ref{1.5}.
For all indices $c,d=1\lcd m$ and each $j\in\ZZ$ the operator
\eqref{ecdj} on the vector space \eqref{w} commutes with any
simultaneous permutation of the variables $x_1\lcd x_N$ and of the $N$ 
tensor factors\/ $\CC^{\ts m}$.
The action of the element $C\in\glhat_m$ on $W$
also commutes with any such permutation. So the actions of the Lie algebra 
$\glhat_m$ and of the symmetric group $\Sym_N$ on $W$ mutually commute.

Any operator \eqref{ecdj} commutes with the multiplications
by $x_1\lcd x_N$ in the tensor factor
$\CC[\ts x_1\ts,x_1^{-1}\lcd x_N\ts,x_N^{-1}\ts]$ of $W$.
Furthermore, in the case $j=0$ the operator \eqref{ecdj}
commutes separately with any permutation of $x_1\lcd x_N$ and
with any permutation of the $N$ tensor factors\/ $\CC^{\ts m}$. 
Hence for $j=0$ it commutes with 
the Cherednik operator \eqref{up} on $W$.
Here we have also used the basic fact that the adjoint action
of $E_{cd}\in\gl_m$ annihilates the element 
$$
\sum_{a,b=1}^m\,E_{ab}\ot E_{ba}\in\gl_m\ot\gl_m\,.
$$

Thus we have proved the part (ii) of Proposition \ref{1.5}.
Moreover, we have proved that for all $c,d=1\lcd m$ and 
every $j\in\ZZ$ the operator \eqref{ecdj} on $W$ commutes with the
action of the subalgebra of $\T_N$ generated by $\Sym_N$ and
$\CC[\ts x_1\ts,x_1^{-1}\lcd x_N\ts,x_N^{-1}\ts]\,$. However, the operator 
\eqref{ecdj} generally does not commute with \eqref{up}. 
In the notation of Subsection \ref{sec:31},
the latter operator for any $p=1\lcd N$
can be written as the composition $X_p\,Y_p\,$.
Since the operator \eqref{ecdj} commutes with $X_p\,$,
we shall consider the commutator with $Y_p\,$.
Moreover, it suffices to consider the case $p=1$ only,
see the very beginning of the present subsection.

Extend the vector space $W$ 
by replacing its first tensor factor
$\CC[\ts x_1\ts,x_1^{-1}\lcd x_N\ts,x_N^{-1}\ts]$ 
by the space of all complex valued 
rational functions in 
$x_1\lcd x_N$ as we did in Subsection~\ref{sec:31}.
Then $Y_1$ can be presented as the restriction 
of the operator $\ka\,D_1-R_{\ts 1}+T_1$ to the space $W$.
The commutator of the summand $\ka\,D_1$ here with
the operator \eqref{ecdj} equals
\begin{equation}
\label{dcom}
\ka\,j\,x_1^{\,j-1}\ot E_{cd}^{\ts(1)}\ot\id\,.
\end{equation}
Further, the operator $R_{\ts 1}$ commutes with \eqref{ecdj}.
The commutator of $T_1$ with \eqref{ecdj} equals
\begin{gather}
\notag
\sum_{r=2}^N\,
\Bigl[\ \frac{1}{x_1-x_r}\ot\si_{1r}\ot\id
\ ,\,
x_1^{\,j}\ot E_{cd}^{\ts(1)}\ot\id+x_r^{\,j}\ot E_{cd}^{\ts(r)}\ot\id
\,\Bigr]\,+
\\
\label{tcom}
\sum_{i=0}^\infty\,
\sum_{a,b=1}^m\,
\bigl[\,\ts 
x_1^{\,-i-1}\ot E_{ab}^{\ts(1)}\ot E_{\ts ba}\,t^{\,i}
\ ,\,
x_1^{\,j}\ot E_{cd}^{\ts(1)}\ot\id+
\id\ot\id^{\,\ot N}\ot E_{\ts cd}\,t^{\,j}
\,\ts\bigr]
\end{gather}
where we have just omitted the zero commutators. Since the 
operator $\si_{1r}$ on $(\CC^{\ts m})^{\ot N}$ can~be written as the sum
$$
\sum_{a,b=1}^m\,E_{ab}^{\ts(1)}\,E_{ba}^{\ts(r)}\,,
$$
the sum in the first of the two displayed lines \eqref{tcom} equals
\begin{equation}
\label{3.7}
\sum_{r=2}^N\,
\sum_{a=1}^m\,
\frac{x_r^{\,j}-x_1^{\,j}}{x_1-x_r}\ot
\bigl(\,
E_{ca}^{\ts(1)}\,E_{ad}^{\ts(r)}-
E_{ad}^{\ts(1)}\,E_{ca}^{\ts(r)}
\,\bigr)
\ot\id\,.
\end{equation}
By using the commutation relations \eqref{hatcom} in $\glhat_m$
and the assumption that the $\glhat_m\ts$-module~$V$ has level $\ka-m\,$,
the sum in the second of the two displayed lines \eqref{tcom} equals
\begin{align*}
\sum_{i=0}^\infty\,
\sum_{a=1}^m\,
\bigl(\,
x_1^{\,j-i-1}\ot 
E_{ad}^{\ts(1)}\ot E_{ca}\,t^{\,i}
&-\,
x_1^{\,j-i-1}\ot
E_{ca}^{\ts(1)}\ot E_{ad}\,t^{\,i}
\\
-\,\,
x_1^{\,-i-1}\ot
E_{ad}^{\ts(1)}\ot E_{ca}\,t^{\,i+j}
&+\,
x_1^{\,-i-1}\ot
E_{ca}^{\ts(1)}\ot E_{ad}\,t^{\,i+j}
\,\bigr)
\end{align*}
$$
+\,
\begin{cases}
\,(m-\ka)\,j\,x_1^{\,j-1}\ot E_{cd}^{\ts(1)}\ot\id&\mbox{if\quad}j<0\,, 
\\ 
\hspace{72pt}
0&\mbox{if\quad}j\ge0\,. 
\end{cases}
$$

Now assume that $j<0\ts$, so that the element
$E_{cd}\,t^{\,j}$ belongs to the subalgebra $\q\subset\glhat_m\,$.
Then by adding the last displayed sum to
the operator \eqref{dcom} and making cancellations we~get
\begin{equation}
\label{3.8}
m\,j\,x_1^{\,j-1}\ot E_{cd}^{\ts(1)}\ot\id\,+
\sum_{i=0}^{-j-1}\,
\sum_{a=1}^m\,
x_1^{\,-i-1}\ot
\bigl(\,
E_{ca}^{\ts(1)}\ot E_{ad}\,t^{\,i+j}-
E_{ad}^{\ts(1)}\ot E_{ca}\,t^{\,i+j}
\,\bigr)\,.
\end{equation}
So for $j<0$ the commutator of the operator
$Y_1$ with \eqref{ecdj} equals the sum of \eqref{3.7}~and~\eqref{3.8}. 
Both the operators \eqref{3.7} and \eqref{3.8}
preserve the space $W$, so its extension is no longer needed.

To finish the proof of part (iii) of Proposition \ref{1.5}, let us 
identify the vector space $W$ with that of the tensor product 
\begin{equation}
\label{mod}
\CC^{\ts m}[\ts t\ts,t^{-1}\ts]^{\ts\ot N}\ot V
\end{equation}
of $\glhat_m\ts$-modules as we 
did in Subsection \ref{sec:13}. 
Denote by $\th$ the representation of the Lie algebra $\glhat_m$ on
the space of \eqref{mod}, so that
$\th\ts(\ts E_{cd}\,t^{\,j}\ts)$ is the operator \eqref{ecdj} 
under the identification~of the latter space with $W$.
For each $r=1\lcd N$ denote by $\th_r$ the representation
of $\glhat_m$ on the $r\,$th tensor factor 
$\CC^{\ts m}[\ts t\ts,t^{-1}\ts]$ of \eqref{mod},
so that under the identification of \eqref{mod} with~$W$
$$
\th_r\ts(\ts E_{cd}\,t^{\,j}\ts)=x_r^{\,j}\ot E_{cd}^{\ts(r)}\ot\id\,.
$$
Further denote by $\th_{N+1}$ the representation
of $\glhat_m$ on the tensor factor $V$ of \eqref{mod}.
Then 
\begin{equation}
\label{theta}
\th=\th_1+\ldots+\th_N+\th_{N+1}\,.
\end{equation}

Introduce the following element of the tensor square of the subalgebra 
$t^{-1}\ts\gl_m[\,t^{-1}\ts]\subset\glhat_m\,$,
$$
J\,=
\sum_{i=0}^{-j-1}\,
\sum_{a=1}^m\,
\bigl(\,
E_{ad}\,t^{\,i+j}\ot E_{ca}\,t^{\,-i-1}-
E_{ca}\,t^{\,i+j}\ot E_{ad}\,t^{\,-i-1}
\,\bigr)\,.
$$
Note that this element is antisymmetric\ts: it belongs to the exterior
square of $t^{-1}\,\gl_m[\,t^{-1}\ts]\,$. 
By rewriting the definition of this element as
\begin{gather*}
J\,=
\sum_{i=0}^{-j-1}\,
\Bigl(\, 
E_{cd}\,t^{\,i+j}\ot
\bigl(\,E_{cc}\,t^{\,-i-1}-E_{dd}\,t^{\,-i-1}\,\bigr)\,+
\\[2pt]
\sum_{a\ne c}\,
E_{ad}\,t^{\,i+j}\ot E_{ca}\,t^{\,-i-1}
-
\sum_{a\ne d}\,
E_{ca}\,t^{\,i+j}\ot E_{ad}\,t^{\,-i-1}
\,\Bigr)
\end{gather*}
we observe that $J$ belongs to the tensor square
of the subalgebra $\q\subset\slhat_m$ and 
thus to $\q\wedge\q\,$.

For every $r=1\lcd N,N+1$ consider
the linear map $\om_{\ts r}:\q\ot\q\to\End W$ 
defined by setting
$\ts\om_{\ts r}\ts(\ts P\ot Q\ts)=\th_r\ts(P)\,\th_1\ts(Q)\ts$
for $P\ts,Q\in\q\,$. Also consider the map
$$
\om=\om_1+\ldots+\om_N+\om_{N+1}\,.
$$
Then
$\ts\om\ts(\ts P\ot Q\ts)=\th\ts(P)\,\th_1(Q)$
due to \eqref{theta}.
Let us apply $\om_1\ts\lcd\ts\om_N\ts,\ts\om_{N+1}$ to $J\,$.
We get
$$
\om_{\ts1}\ts(J)=
j\,x_1^{\,j-1}\ot
(\,
m\,E_{cd}^{\ts(r)}-\de_{cd}\,\ts\id^{\,\ot N}
\,)\ot\id\,.
$$
For the indices $r=2\lcd N$ we get
$$
\om_{\ts r}\ts(J)\,=
\sum_{i=0}^{-j-1}\,
\sum_{a=1}^m\,
x_1^{\,-i-1}\,x_r^{\,i+j}\,
\bigl(\,
E_{ca}^{\ts(1)}\,E_{ad}^{\ts(r)}-
E_{ad}^{\ts(1)}\,E_{ca}^{\ts(r)}
\,\bigr)
$$
which coincides with the sum \eqref{3.7} due to the identity 
$$
\frac{v^{\,j}-u^{\,j}}{u-v}\,=\,
\sum_{i=0}^{-j-1}
u^{\,-i-1}\,v^{\,i+j}
\quad\text{for}\quad j<0\,.
$$
Further,
$$
\om_{\ts N+1}\ts(J)\,=
\sum_{i=0}^{-j-1}\,
\sum_{a=1}^m\,
x_1^{\,-i-1}\ot
\bigl(\,
E_{ca}^{\ts(1)}\ot E_{ad}\,t^{\,i+j}-
E_{ad}^{\ts(1)}\ot E_{ca}\,t^{\,i+j}
\,\bigr)
$$
which coincides with the sum \eqref{3.8} less its first summand. 
Hence for $j<0$ the commutator of the operator
$Y_1$ with \eqref{ecdj} equals 
$$
\de_{cd}\,\ts j\,x_1^{\,j-1}\ot\id^{\,\ot N}\ot\id+
\om_1(J)+\ldots+\om_N(J)+\om_{N+1}(J)
\,=\,
\de_{cd}\,\ts j\,x_1^{\,j-1}\ot\id^{\,\ot N}\ot\id+\om(J)\,.
$$
By linearity of the map \eqref{theta}
it follows that for any given element $P\in\q\,$, the commutator 
$[\,Y_1\,,\th\ts(P)\,]$ belongs to the right ideal of the
algebra $\End W$ generated by the image $\th\ts(Q)$ of a certain element 
$Q\in\q$ depending on $P\,$. 
Hence the operator $Y_1$ preserves the subspace
$\q\,W\subset W$. Thus we complete the
proof of the part (iii) of Proposition~\ref{1.5}.

By inspecting the arguments given in this subsection,
we also get the parts (ii) and (iii) of Corollary~\ref{1.6}. Indeed, 
if we modify the definitions of $Y_p$ and $T_p$ given in the beginning 
of Subsection \ref{sec:32} by subtracting \eqref{subtr}, then
for $p=1$ we will have to subtract from \eqref{tcom}
$$
\frac1m\,\ts\sum_{i=0}^\infty\,
\bigl[\,\ts 
x_1^{\,-i-1}\ot\id^{\,\ot\ts N}\ot I\,t^{\,i}
\ ,\,
\id\ot\id^{\,\ot N}\ot E_{\ts cd}\,t^{\,j}
\,\ts\bigr]\,.
\vspace{-8pt}
$$
By using the commutation relations \eqref{hatcom} in $\glhat_m$
and the assumption that the $\glhat_m\ts$-module~$V$ has level $\ka-m\,$,
in the case of $j<0$ the above displayed expression equals
$$
\frac1m\,\ts\de_{cd}\,
(m-\ka)\,j\,x_1^{\,j-1}\ot\id^{\,\ot N}\ot\id
$$
which has to be then subtracted from \eqref{3.8}. 
By linearity, this modification has no effect
when considering the commutator
$[\,Y_1\,,\th\ts(P)\,]$ for any $P\in\q\,$.
Hence for the modified operator $Y_1\,$, the
$[\,Y_1\,,\th\ts(P)\,]$ still belongs to the right ideal of
$\End W$ generated by $\th\ts(Q)$ for some $Q\in\q\,$.



\end{document}